\documentclass[11pt]{amsart}
\usepackage{graphics}
\usepackage{color}
\usepackage{extarrows}

\usepackage{lscape}
\usepackage{latexsym}
\usepackage{amsmath,verbatim}
\usepackage{amsthm}
\usepackage{amssymb}
\usepackage[mathscr]{euscript}
\usepackage{yfonts}
\usepackage{stmaryrd}
\usepackage{multicol}
\usepackage{bm}
\usepackage[all]{xy}
\numberwithin{equation}{section}
\newtheorem{thm}{Theorem}[section]
\newtheorem{prop}[thm]{Proposition}
\newtheorem{lem}[thm]{Lemma}
\newtheorem{cor}[thm]{Corollary}
{\bf}{\it}

\newtheorem{fthm}{Theorem}{\bf}{\it}
{\bf}{\it}
\newtheorem{fcor}[fthm]{Corollary}{\bf}{\it}
\newtheorem{fconj}[fthm]{Conjecture}{\bf}{\it}
{\bf}{\it}

\theoremstyle{definition}
\newtheorem{defn}[thm]{Definition}

{\bf}{\rm}

\theoremstyle{remark}
\newtheorem{ex}[thm]{Example}
\newtheorem{rem}[thm]{Remark}
{\bf}{\it}

\newtheorem{definition and corollary}[thm]{Definition and Corollary}

{\it}{\rm}

\newcommand{\tbB}{\widetilde{\mathbf B}}
\newcommand{\tbG}{\widetilde{\mathbf G}}

\newcommand{\al}{\alpha}
\newcommand{\af}{\mathrm{af}}

\newcommand{\tbP}{\widetilde{\mathbf P}}
\newcommand{\C}{{\mathbb C}}

\newcommand{\cO}{{\mathcal O}}
\newcommand{\bG}{{\mathbf G}}
\newcommand{\Par}{{\mathcal P}}

\newcommand{\tth}{{\mathtt h}}

\newcommand{\gdim}{\mathrm{gdim}}
\newcommand{\ro}{\mathrm{rot}}
\newcommand{\ce}{\mathrm{c}}

\newcommand{\la}{\lambda}
\newcommand{\La}{\Lambda}

\newcommand{\Gr}{\mathrm{Gr}}

\newcommand{\Sym}{\mathfrak{S}}
\newcommand{\tSym}{\widetilde{\mathfrak{S}}}

\newcommand{\MID}{\! \! \mid}

\newcommand{\g}{\mathfrak{g}}

\newcommand{\tT}{\widetilde{T}}

\newcommand{\tI}{\mathtt{I}}

\newcommand{\tP}{\mathtt{P}}
\newcommand{\sP}{\mathsf{P}}

\newcommand{\gn}{\mathfrak{n}}

\newcommand{\bv}{\mathbf{v}}

\newcommand{\bO}{\mathbb{O}}

\renewcommand{\P}{\mathbb{P}}

\newcommand{\sX}{\mathscr{X}}

\newcommand{\Z}{\mathbb{Z}}

\newcommand{\Ga}{\mathbb G_a}
\newcommand{\Gm}{\mathbb G_m}

\title[A geometric realization of Chromatic symmetric functions]{A geometric realization of the chromatic symmetric function of a unit interval graph}

\author{Syu \textsc{Kato}}

\address{
Department of Mathematics, Kyoto University, Oiwake Kita-Shirakawa Sakyo Kyoto 606-8502 JAPAN}
\email{syuchan@math.kyoto-u.ac.jp}

\begin{document}

\begin{abstract} 
Shareshian-Wachs, Brosnan-Chow, and Guay-Pacquet [Adv. Math. {\bf 295} (2016), {\bf 329} (2018), arXiv:1601.05498] realized the chromatic (quasi-)symmetric function of a unit interval graph in terms of Hessenberg varieties. Here we present another realization of these chromatic (quasi-)symmetric functions in terms of the Betti cohomology of the variety $\mathscr X_\Psi$ defined in [arXiv:2301.00862]. This yields a new inductive combinatorial expression of these chromatic symmetric functions. Based on this, we propose a geometric refinement of the Stanley-Stembridge conjecture, whose validity would imply the Shareshian-Wachs conjecture.
\end{abstract}
\maketitle

\section*{Introduction}
A unit interval graph $\Gamma$ is a graph induced by the overlapping pattern of finite segments of same size in $[0,1]$ (without covering relations). Brosnan-Chow \cite{BC18} and Guay-Pacquet \cite{GP16} connected the Betti cohomology of Hessenberg varieties attached to $\Gamma$ with the chromatic quasisymmetric function $X_\Gamma (t)$ of $\Gamma$ defined by Shareshian-Wachs \cite{SW16}, which reduces to Stanley's chromatic symmetric function $X_\Gamma$ \cite{Sta95} upon $t=1$. This opens a way to understand the context of the Stanley-Stembridge conjecture \cite{SS93} on the $e$-positivity of the chromatic symmetric function of a $(3+1)$-free poset using geometry since Guay-Pacquet \cite{GP13} reduces the Stanley-Stembridge conjecture to these cases. This development attracts a lot of attention (see e.g. \cite{CHSS,AP18,HP19,AB21,AB22,SP,VX21,AN22a,AN22b}). Although the Stanley-Stembridge conjecture itself was recently settled by Hikita \cite{Hik24a}, we still do not fully understand all the questions arising from this line of research, including its graded version (the Shareshian-Wachs conjecture \cite[Conjecture 1.3]{SW16}).

In this paper, we realize the chromatic symmetric function of a unit interval graph using the variety $\mathscr X_\Psi$ attached to a Dyck path defined in \cite{Kat23a} and discuss its consequences. Roughly speaking, we have correspondences
\begin{equation}
\{\text{unit interval graphs}\} \Leftrightarrow \{\text{Hessenberg functions}\}\Leftrightarrow \{\text{Dyck paths}\}.\label{eqn:introcorr}
\end{equation}
We will write $\mathscr X_\Gamma$ instead of $\mathscr X_\Psi$ when $\Gamma$ corresponds to $\Psi$ through (\ref{eqn:introcorr}).

Let $\La$ be the ring of symmetric functions. Fix $n \in \Z_{>0}$ and set $G:=\mathop{GL}(n,\C)$, $G[\![z]\!] := \mathop{GL}(n,\C[\![z]\!])$. Let $\tP^+$ be the set of partitions of length $\le n$. For each (labelled) unit interval graph $\Gamma$ with $n$ intervals, there exists a smooth projective algebraic variety $\mathscr X_\Gamma$ (defined in \cite{Kat23a}) equipped with a $(G[\![z]\!] \times \C^{\times})$-equivariant map
$$\mathsf{m}_{\Gamma} : \mathscr X_\Gamma \to \Gr,$$
where $\Gr$ is the affine Grassmannian of $\mathop{GL}(n)$.

For each $\la \in \tP^+$, we have a $G[\![z]\!]$-orbit $\mathbb O_\la \subset \Gr$ that yields the intersection cohomology complex $\mathsf{IC}_\la$ (see \cite{BBDG,Lus81a}). The geometric Satake correspondence \cite{MV07} equips $H^{\bullet} ( \mathsf{IC}_\la )$ with the structure of a $\mathop{GL}(n,\C)$-module whose character is the Schur function $s_\la \in \La$.

\begin{fthm}[$\doteq$ Theorem \ref{thm:main}]\label{fthm:real}
We have $( \mathsf{m}_\Gamma )_*\C_{\mathscr X_\Gamma} \cong \bigoplus _{\la \in \tP^+} V^{\la} (\mathscr X_\Gamma) \boxtimes \mathsf{IC}_\la$, where $V^\la (\mathscr X_\Gamma) \in D^b(\mathrm{pt})$ is a $($graded$)$ vector space. Using this, we have
$$X_\Gamma (t) = t^{-d_\Gamma}\sum_{\la \in \tP^+} s_\la \cdot  t^{i/2} \dim \, H^{i} ( V^\la (\mathscr X_\Gamma) ) \hskip 5mm \exists d_{\Gamma} \in \Z_{\ge 0}.$$
\end{fthm}

Theorem \ref{fthm:real} does not involve the Frobenius characteristic as in \cite{SW16,BC18}. In this sense, we might say that Theorem \ref{fthm:real} is a more direct realization of $X_\Gamma$ compared with [{\it loc. cit.}].

Let $P_\af$ be the weight lattice of the affinization of $\mathfrak{gl}(n)$ generated by $\varLambda_0,\ldots,\varLambda_n$ and $\delta$. Let $\tSym$ be the Weyl group of type $\mathsf{A}_{n-1}^{(1)}$ with its generators $\{s_i\}_{i=1}^n$, that acts on $\C P_\af$. We have a projection
$$\mathsf{pr}: \C P_{\af} \stackrel{e^{\varLambda_0}, e^{\delta} \mapsto 1}{\longrightarrow} \C [X_1^{\pm 1},\ldots,X_n^{\pm 1}]^{\Sym_n} = \{\text{space of characters of }\mathop{GL}(n,\C)\}.$$
Let $X_{\Gamma}^{(n)}$ be the image of $X_{\Gamma} \in \La$ under the truncation map $\La \rightarrow \Z [X_1,\ldots,X_n]$.

For each $0 \le h \le i \le n$, we set
$$\mathtt{S}_{i,h} := (s_{h-1} \cdots s_1 s_n \cdots s_{i+1} s_i + s_{h-2} \cdots s_n \cdots s_i + \cdots + s_{i+1}s_i+ s_i + 1 ) \in \Z \tSym_.$$
Applying the localization theorem to Theorem \ref{fthm:real}, we deduce:

\begin{fthm}[$\doteq$ Theorem \ref{thm:alg}]\label{fthm:char}
For each unit interval graph $\Gamma$, we have a sequence $h_1=0\le h_2 \le \cdots \le h_n < n$ $($specified in $\S \ref{subsec:statement})$ such that
$$X_\Gamma^{(n)} = \mathsf{pr}\Bigl( \left( \mathtt{S}_{1,0} \cdot \mathtt{S}_{2,h_2} \cdot \cdots \cdot \mathtt{S}_{n,h_n} \right) e^{\varLambda_n} \Bigr) \in \Z [X_1,\ldots,X_n].$$
\end{fthm}

An algebraic explanation of Theorem \ref{fthm:char} is forthcoming \cite{Hik24b}. As a consequence of Theorem \ref{fthm:char}, we have:

\begin{fcor}[$\doteq$ Corollary \ref{cor:EP}, see e.g. \cite{DK22}]\label{fcor:EP}
For each unit interval graph $\Gamma$ with $n$-vertices, the Euler-Poincar\'e characteristic $\chi ( \mathscr X_\Gamma )$ of $\mathscr X_\Gamma$ is the number of proper colorings of $\Gamma$ with $n$-colors.
\end{fcor}

These analysis motivate us to propose the following:

\begin{fconj}[Geometric Stanley-Stembridge conjecture]\label{fconj:GSS}
For a proper smooth variety $\mathscr X$ equipped with a $G[\![z]\!]$-action and a $G[\![z]\!]$-equivariant map $f : \mathscr X \longrightarrow \Gr$, we decompose as:
$$f_* \C \cong \bigoplus_{\la \in \tP^+} V^{\la} (\mathscr X) \boxtimes \mathsf{IC}_{\la} \in D_{c}^b ( \Gr ), \hskip 5mm \text{where} \hskip 5mm V^{\la} (\mathscr X) \in D_{c}^b ( \mathrm{pt} ).$$
Then, the symmetric functions
$$\sum_{\la \in \tP^+} s_\la \cdot \dim \, H^i ( V^\la(\mathscr X) ) \in \La \hskip 10mm i \in \Z$$
expand positively with respect to the elementary symmetric functions.
\end{fconj}

The Shareshian-Wachs conjecture follows if Conjecture \ref{fconj:GSS} holds for all $\mathscr X_\Psi$. Hence, Conjecture \ref{fconj:GSS} offers a possible interpretation that the Stanley-Stembridge conjecture reflects some geometric structure of $\Gr$, (an algebro-geometric realization of) the classifying space of $U(n)$.

The organization of this paper is follows: We prepare basic material in section one that includes recollections (and slight combinatorial reformulations) of \cite{Kat23a}, a characterization of chromatic quasisymmetric functions \cite{AB21}, and the theory of character sheaves \cite{Lus85}. In section two, we first state our main results (Theorem \ref{fthm:real} and \ref{fthm:char}), provide examples and alternative proofs of well-known properties of chromatic quasisymmetric functions of unit interval graphs (Corollary \ref{cor:qs}), and then prove our main results.

\section{Preliminaries}\label{sec:prelim}
We always work over $\C$ unless otherwise specified. A variety is a separated normal scheme of finite type over $\C$, and we may not distinguish a variety $\sX$ with its set of $\C$-valued points. In particular, we have $\Gm = \C^{\times}$ and $\Ga = \C$. We set $[n]:=\{1, 2,\ldots, n\}$. For a complex $( C^{\bullet}, d )$ of $\C$-vector space, we set
$$\dim \, H^{\bullet}( C^{\bullet}, d ) := \sum_{i \in \Z} \dim H^{i}( C^{\bullet}, d ) \hskip 5mm \text{and} \hskip 5mm \gdim \, H^{\bullet}( C^{\bullet}, d ) := \sum_{i \in \Z} t^{i/2} \dim H^{i}( C^{\bullet}, d ).$$
We set $[n]_t:=\frac{1-t^n}{1-t}$ and $[n]_t! := [n]_t[n-1]_t\cdots [1]_t$ for each $n \in \Z_{\ge 0}$.

\subsection{Algebraic Groups}
Here conventions are largely in common with \cite{Kat23a}, see also \cite{Kum02,CG97}. Fix $n \in \Z_{>0}$ and consider
$$G = \mathop{GL} ( n ) \subset M_n \cong \C^{n^2} \hskip 3mm \text{and} \hskip 3mm \bG = \mathop{GL} (n, \C [\![z]\!]) \subset G(\!(z)\!) := \mathop{GL} (n, \C (\!(z)\!)).$$
Let $E_{ij} \in M_n$ ($1 \le i,j \le n$) be the matrix unit. We have an extra loop rotation $\Gm$-action on $G(\!(z)\!)$, and the group $\Gm \times G(\!(z)\!)$ admits a central extension $\widetilde{G}(\!(z)\!)$ by $\Gm$. We denote by $\Gm^{\ro}$ and $\Gm^{\ce}$ the subtori of $\widetilde{G}(\!(z)\!)$ corresponding to the loop rotation and the center. Let $T \subset G$ be the diagonal torus and let $B \subset G$ be the upper triangular part of $G$. We have the evaluation map
$$\mathtt{ev}_0 : \tbG \longrightarrow G \hskip 10mm \Gm^{\ro} \mapsto \{e\}, \hskip 3mm z \mapsto 0.$$
We set $\tbB := \mathtt{ev}_0^{-1} ( B )$. For each $1 \le i < n$, we set $P_i \subset G$ (resp. $\tbP_i \subset \bG$) as the (algebraic or pro-algebraic) subgroup generated by $B$ (resp. $\tbB$) and $\mathrm{Id} + \C E_{i+1,i}$ inside $G$ (resp. $\tbG$). We set $\tbP_0$ as the (pro-)algebraic group generated by $\tbB$ and $\mathrm{Id} + \C z^{-1} E_{1,n}$ inside $\widetilde{G}(\!(z)\!)$. We define the extended torus as $\tT := T \times \Gm^\ro \times \Gm^\ce \subset \tbG$. We have $\tbP_i \cap \tbP_j = \tbB$ when $i \neq j$.

For $1 \le i \le n$, we have a character $\varepsilon_i : T \rightarrow \Gm$ that extracts the $i$-th entry of $T$. We set
$$\sP := \bigoplus_{i=1}^n \Z \varepsilon_i$$
and consider its subsets
$$\tP := \sum_{i=1} \Z_{\ge 0} \varepsilon_i \hskip 2mm\text{and} \hskip 2mm \sP^+ := \{\sum_{i=1}^n \la_i \varepsilon_i \in \sP \mid \la_1 \ge \la_2 \ge \cdots \ge \la_n \}.$$
For $\la = \sum_{i=1}^n \la_i \varepsilon_i \in \sP$, we set $|\la| := \sum_{i=1}^n \la_i \in \Z$. The permutation of indices define $\mathfrak S_n$-actions on $\sP$ and $\tP$. We set $\tP^+ := ( \sP^+ \cap \tP )$ and identify it with the set of partitions with its length $\le n$. We set
$$\tP_{n} := \{\la \in \tP \mid |\la|=n\} \hskip 5mm \text{and} \hskip 5mm \tP_{n}^+:=\tP_n \cap \tP^+.$$
We may regard $\varepsilon_i$ as a character of $\widetilde{T}$ through the projection to $T$. Let $\varLambda_0$ and $\delta$ denote the degree one characters of $\Gm^\ce$ and $\Gm^\ro$ extended to $\widetilde{T}$ trivially, respectively. We set $\varpi_n := \sum_{j=1}^n \varepsilon_j$ and
$$\varLambda_i := \varLambda_0 + \sum_{j=1}^i \varepsilon_j \hskip 3mm (1 \le i < n), \hskip 5mm \varLambda_0 + \varpi_n \hskip 3mm (i = n).$$
We frequently identify the index $0$ with $n$ in the sequel. We set
$$\sP_\af := \Z \delta \oplus \Z \varLambda_0 \oplus \bigoplus_{i \in [n]} \Z \varepsilon_i  \hskip 3mm \text{and} \hskip 3mm \sP^+_\af :=  \Z \varpi_n \oplus \Z \delta \oplus \sum_{i \in [n]} \Z_{\ge 0} \varLambda_i \subset \sP_\af.$$
We set $\al_{i,j}:= \varepsilon_i - \varepsilon_j$ for $1 \le i < j \le n$ and $\Delta^+ := \{ \al_{i,j} \}_{1 \le i < j \le n} \subset \sP$. We put $\al_i := \al_{i,i+1}$ for $1 \le i < n$ and $\al_0 := - \al_{1,n} + \delta$. For $\la,\mu \in \tP_n^+$, we set $\la \unrhd \mu$ if $\la \in \mu + \sum_{\al \in \Delta^+} \Z_{\ge 0} \al$. We define an inner form on $\sP_\af$ as $\left< \varepsilon_i, \varepsilon_j \right> = \delta_{ij}$ and $\varLambda_0, \delta \in \mathrm{Rad} \, \left< \bullet, \bullet \right>$. For $\al = \al_{i,j} \in \Delta^+$, we set $\g_\al := \C E_{ij} \subset \gn := \mathrm{Lie} \, [B,B] \subset M_n$.

The affine Weyl group $\tSym_n$ of $\mathsf{A}_{n-1}$ is generated by $\{s_i\}_{i \in [n]}$. We have an action of $\tSym_n$ on $\sP_\af$ as:
$$s_i ( \varLambda ) := \varLambda - ( \left< \al_i, \varLambda \right> + \delta_{in} \varLambda ( K ) )\al_i \hskip 5mm i \in [n], \varLambda \in \sP_\af$$
where $K \in \mathrm{Hom} ( \sP_\af, \Z )$ satisfies $\varpi_i ( K ) = 0$ $(i \in \tI_\af)$, $\delta ( K ) = 0$, and $\varLambda_0 ( K ) = 1$.

For each $i \in ([n]\cup\{0\})$, let $L ( \varLambda _i )$ be the integrable highest weight module of an affine Lie algebra of $\mathfrak{gl} ( n,\C )$ (on which $G(\!(z)\!) \times \Gm$ acts by \cite[Chap. V\!I]{Kum02}) with a fixed $\tbB$-eigenvector $\bv_{\varLambda_i}$.

\subsection{Combinatorial setup}

\begin{defn}[Hessenberg function]
Let $n \in \Z_{>0}$. A Hessenberg function $\tth:[n]\to [n]$ (of size $n$) is a map satisfying $\tth(1)\le \tth(2)\le \cdots \le \tth(n-1) \le \tth(n)=n$ and $\tth(i) \ge i$ for all $i \in [n]$. Let $\mathbb H_n$ be the set of Hessenberg functions of size $n$, and we set $\mathbb H := \bigsqcup_{n > 0} \mathbb H_n$.
\end{defn}
\begin{defn}[Unit interval graph]
For a Hessenberg function $\tth \in \mathbb H_n$, let $\Gamma_\tth$ denote the graph whose vertex set is $[n]$, and we have an edge joining $i$ and $j$ $(i < j)$ if and only if $i < j \le f(i)$. Such a graph is called a unit interval graph.
\end{defn}

\begin{defn}[Root ideals]
A subset $\Psi \subset \Delta^+$ is called a root ideal if and only if $( \Psi + \Delta^+ ) \cap \Delta^+ \subset \Psi$. It is easy to find that $\Psi$ is a root ideal if and only if $\al_{i,j}\in \Psi$ implies $\al_{i',j}, \al_{i,j'} \in \Psi$ for $i' < i$ and $j < j'$. For a root ideal $\Psi \subset \Delta^+$, we define
$$\gn ( \Psi ) : = \bigoplus _{\al \in \Psi} \g_\al \subset \gn.$$
For a root ideal $\Psi \subset \Delta^+$ and $1 \le i \le n$, we define $$h_i (\Psi) := \max \left( \{ j\mid \al_{ji} \in \Psi \} \cup \{0\}\right).$$
\end{defn}

\begin{lem}[Cellini \cite{Cel00} \S 3]\label{lem:rootid}
For a root ideal $\Psi \subset \Delta^+$, the subspace $\gn ( \Psi ) \subset \gn$ is $B$-stable.\hfill $\Box$
\end{lem}

The following result is modified from the original by changing $\gn ( \Psi )$ with $\gn(\Psi)^{\perp}$, where $\perp$ is taken with respect to the invariant bilinear form on $\mathfrak{gl}(n)$:

\begin{prop}[De Mari-Procesi-Shayman \cite{MPS92} Lemma 1]
For a Hessenberg function $\tth$, we define a root ideal $\Psi_\tth$ as $\al_{i,j} \in \Delta^+ \setminus \Psi_\tth$ if and only if $i < j \le \tth(i)$. For a root ideal $\Psi \subset \Delta^+$, we define the Hessenberg function $\tth_\Psi$ as $\tth_{\Psi}(i) := \max \bigl( \{j \mid \al_{i,j} \not\in \Psi \} \cup \{ i \} \bigr)$. This sets up a bijection between the Hessenberg functions and root ideals. \hfill $\Box$
\end{prop}

We have an involution on the set of set of root ideals induced by an endomorphism $\theta$ of $\Delta^+$ given as: $\al_{ij} \mapsto \al_{n+1-j,n+1-i}$ ($1 \le i < j \le n$).

By construction, we have $\tth_\Psi ( n+1-i ) + h_i ( \theta ( \Psi ) ) = n$ for each  $1 \le i \le n$.

\subsection{Chromatic symmetric functions and the modular law}
For generalities on symmetric functions, see \cite{Mac95}. Let $\La$ be the space of symmetric functions over $\C(t^{1/2})$, and let $\Par$ be the set of all partitions. Let $\{m_\la\}_{\la \in \Par}$, $\{s_\la\}_{\la \in \Par}$, and $\{e_\la\}_{\la \in \Par}$ denote the $\C(t^{1/2})$-bases of $\La$ offered by monomial symmetric functions, Schur functions, and elementary symmetric functions, respectively. We have a truncation map
$$\pi_n : \La \longrightarrow \La^{(n)} := \C(t^{1/2}) [X]^{\Sym_n} \subset \C(t^{1/2}) [X] = \C(t^{1/2}) [X_1,\ldots,X_n]$$
by setting $0 = X_{n+1} = X_{n+2} = \cdots$. The images of $m_\la$, $s_\la$, and $e_\la$ in $\La^{(n)}$ are $0$ if and only if $\la_{n+1} \neq 0$, and the non-zero images of $\{m_\la\}_{\la \in \Par}$, $\{s_\la\}_{\la \in \Par}$, and $\{e_\la\}_{\la \in \Par}$ in $\La^{(n)}$ define bases of $\La^{(n)}$, that we denote by the same letter for brevity. For each $\gamma = \{ \gamma_i \}_i \in \tP^+$, we set $X^\gamma := X_1^{\gamma_1} \cdots X_n^{\gamma_n} \in \C(t^{1/2})[X]$.

\begin{thm}[Abreu-Nigro \cite{AB21} Theorem 1.1, cf. \cite{AS20} (12)]\label{thm:AN}
The chromatic symmetric function $X_{\bullet}(t) :\mathbb H \to \La$ is a unique function with the following properties:
\begin{enumerate}
\item $(\mathrm{Modular\ law})$ Assume that $\tth_0,\tth,\tth_1 \in \mathbb H_n$ correspond to the root ideals $\Psi_0 \subsetneq \Psi \subsetneq \Psi_1 \subset \Delta^+$, respectively. Let $1 \le i < n$ and suppose that $\Psi_0,\Psi_1$ are $s_i$-stable. We have
$$(1+t) X_\tth(t) = X_{\tth_0}(t) + t X_{\tth_1}(t)$$
if $|\Psi_0| + 1 = |\Psi| = | \Psi_1 | - 1$ holds, and either of the following two conditions are satisfied:
\begin{itemize}
\item We have $\al_{i,j} \in (\Psi \setminus \Psi_0)$, $\al_{(i+1),j} \in (\Psi_1 \setminus \Psi)$ for some $i < j < n$; 
\item We have $\al_{j,i} \in (\Psi \setminus \Psi_0)$, $\al_{j,(i+1)} \in (\Psi_1 \setminus \Psi)$ for some $j < i < n$;
\end{itemize}
\item $(\mathrm{Factorization\ law})$ If $\Gamma_\tth = \Gamma_{\tth_0} \sqcup \Gamma_{\tth_1}$ for $\tth,\tth_0,\tth_1 \in \mathbb H$, then $X_\tth (t) = X_{\tth_0} (t) \cdot X_{\tth_1} (t)$;
\item If $\Gamma_\tth$ is a complete graph of $n$-vertices, then we have $X_\tth (t) = [n]_t! \cdot e_n$.
\end{enumerate}
\end{thm}

\subsection{Geometry of $\mathscr X_\Psi$}

We set $X := G / B$ and call it the flag manifold of $G$. We define a $G$-equivariant vector subbundle
$$T_{\Psi}^* X := G \times^B \gn ( \Psi ) \subset G \times^B \gn \cong T^* X \hskip 10mm \text{for a root ideal} \hskip 3mm \Psi \subset \Delta^+.$$
Let $\pi_\Psi: T^*_\Psi X \to X$ be the projection map. By convention, if $\Psi$ is replaced with a root ideal of $\mathop{GL}(m)$, then we understand that $G=\mathop{GL}(m)$ and $X$ is the flag variety of $\mathop{GL}(m)$.

We fix a root ideal $\Psi \subset \Delta^+$, and consider pro-algebraic groups
$$\tbG_i := \left< \tbP_j \mid i \le j < n \text{ or } 0 \le j < h_i (\Psi)\right>\subset \widetilde{G}(\!(z)\!) \hskip 3mm 1 \le i \le n$$
and $\tbG_{n+1} := \tbG$. Note that each $\tbG_i$ contains $\tbB$, and the quotient $\tbG_i / \tbB$ is finite-dimensional since $h_i ( \Psi ) < i$.

Then, we define $\mathscr X_\Psi ^{(n+1)} := \mathrm{pt}$ and define $\mathscr X_{\Psi}^{(j)}$ ($1 \le j \le n$) by downward induction as follows:
\begin{equation}
\mathscr X^{(j)}_\Psi := \tbG_j ( [\bv_{\varLambda_j}] \times \mathscr X^{(j+1)}_\Psi ) \hookrightarrow \P ( L ( \varLambda_j ) ) \times \prod_{k=j+1}^n \P ( L ( \varLambda_k ) ).\label{eqn:ind-def}
\end{equation}
We set $\mathscr X_\Psi:= \mathscr X^{(1)}_\Psi$. By (\ref{eqn:ind-def}) and $\tbG_1 = \tbG$, we have a $\tbG$-equivariant proper morphism
$$\mathsf{m}_\Psi : \mathscr X_\Psi \longrightarrow \P ( L (  \varLambda_n  ))) \cong \P ( L (  \varLambda_0  ))$$
that lands on the affine Grassmannian $\Gr$ of $G$.

\begin{thm}[Kato \cite{Kat23a} Theorem 3.10, Corollary 4.10, and Corollary 3.12]\label{thm:X-Psi}
Let $\mathcal N \subset M_n$ be the nilpotent cone. For a root ideal $\Psi \subset \Delta^+$, it holds:
\begin{enumerate}
\item We have a $(G \times \Gm)$-equivariant open embedding $T^*_\Psi X \subset \mathscr X_\Psi$. This induces an embedding $\mathscr X_{\Psi} \subset \mathscr X_{\Psi'}$ whenever $\Psi \subset \Psi' \subset \Delta^+$;
\item We have a commutative diagram:
$
\xymatrix{
T^*_\Psi X \ar[r]^{\mu_\Psi} \ar@{^{(}->}[d] & \mathcal N \ar@{^{(}->}[d]&\\
\mathscr X_\Psi \ar[r]^{\mathsf{m}_\Psi} & \Gr \ar@{^{(}->}[r]& \P ( L ( \varLambda_n ) )
}.
$
\end{enumerate}
\end{thm}

We set $\Gr^{(n)} := \mathrm{Im}\, \mathsf{m}_{\Delta^+} \subset \Gr$. Thanks to \cite{Lus81a}, the set of $(T\times \Gm)$-fixed points of $\Gr^{(n)}$ is in bijection with $\gamma \in \tP_n$, and the set of $\tbG$-orbits of $\Gr^{(n)}$ is in bijection with $\tP^+_n$. Let $\imath_\gamma : [\gamma] \hookrightarrow \Gr$ denote the corresponding inclusion. Let $\bO_\la \subset \Gr$ be the $\tbG$-orbit corresponding to $\la \in \tP^+_n$. We have the intersection cohomology complex $\mathsf{IC}_\la$ supported on the closure of $\bO_\la$.

\begin{cor}\label{cor:Pdec}
In the setting of Theorem \ref{thm:X-Psi}, we have $\tbG \cdot T^*_\Psi X = \mathscr X_\Psi$ and $\mathrm{Im}\, \mathsf{m}_\Psi \subset \Gr^{(n)} = \tbG \cdot \mathcal N$.
\end{cor}

\begin{proof}
We prove the first assertion. The $\Gm$-action attracts the Zariski dense subset $T^*_\Psi X \subset \mathscr X_\Psi$ to $X$ \cite[Theorem 3.10]{Kat23a}. Thus, it suffices to find that $X$ is contained in every $G[\![z]\!]$-orbit closure of $\mathscr X_\Psi$. For a collection $\{[\bv_j]\}_j \in \prod_{j=1}^n \P ( L ( \varLambda_j ))$ of point, we can find $g \in G[\![z]\!]$ such that $g \bv_j \in L ( \varLambda_j )$ contains a non-zero $\Gm$-degree zero component. It follows that $\{[g\bv_j]\}_j \in T^*_\Psi X$, that implies that every $G[\![z]\!]$-orbit contains $X$ in its closure.

The second assertion is a consequence of Theorem \ref{thm:X-Psi} 1).
\end{proof}

\begin{cor}\label{cor:dec}
Keep the setting of Theorem \ref{thm:X-Psi}. We have
$$(\mathsf{m}_\Psi)_* \C_{\mathscr X_\Psi} \cong \bigoplus_{\la \in \tP_n^+} V^\la_\Psi \boxtimes \mathsf{IC}_\la, \hskip 5mm \text{where $V^\la_\Psi \in D^b_c(\mathrm{pt})$ is a graded vector space}.$$
\end{cor}

\begin{proof}
Apply the decomposition theorem \cite[Th\'eor\`eme 6.2.5]{BBDG} to $\mathsf{m}_\Psi$.
\end{proof}

For each $\la \in \tP^+_n$, let $\cO_{\la}$ denote the $G$-orbit of $\mathcal N$ corresponding to the nilpotent Jordan blocks of type $\la$. It carries the intersection cohomology complex $\mathsf{IC}_\la^{\mathrm{Spr}}$ supported on $\overline{\cO}_\la$ (see e.g. \cite{BM81}).

\begin{thm}[Lusztig]\label{thm:lusztig}
For $\la \in \tP^+_n$, it holds:
\begin{enumerate}
\item We have $s_\la = \sum_{\gamma \in \tP_n} X^{\gamma} \dim \, H^{\bullet} (\imath_{\gamma}^* \mathsf{IC}_\la ) \in \La^{(n)}$;
\item We have $\dim \, H^{\bullet} ( \Gr, \mathsf{IC}_\la) = \sum_{\gamma \in \tP_n} \dim \, H^{\bullet} (\imath_{\gamma}^* \mathsf{IC}_\la )$;
\item For each $\la \in \tP^+_n$, the restriction of $\mathsf{IC}_\la$ to $\mathcal N$ yields $\mathsf{IC}^{\mathrm{Spr}}_\la$.
\end{enumerate}
\end{thm}

\begin{proof}
By \cite[Theorem 6.1]{Lus81a} and \cite[Theorem 4.1]{Lus81a}, we find
$$s_\la = \sum_{\mu \in \tP_n^+} m_\mu \dim \, H^{\bullet} (\imath_{\mu}^* \mathsf{IC}_\la ).$$
We have $H^{\bullet} (\imath_{\gamma}^* \mathsf{IC}_\la ) = H^{\bullet} (\imath_{u\gamma}^* \mathsf{IC}_\la )$ for each $\mu \in \tP_n^+$ and $u \in \Sym_n$ by the action of $\Sym_n \subset G[\![z]\!]$. This yields the first assertion. The second assertion is follows from \cite[\S 11]{Lus81a}. The last assertion follows from \cite[\S 2]{Lus81b}.
\end{proof}

\subsection{Inductions of character sheaves}\label{subsec:ind}
Let $\mathcal M_{n}$ denote the category of $\mathop{GL}(n,\C)$-equivariant perverse sheaves on the nilpotent cone $\mathcal N_n = \mathcal N$ of $G$ \cite[\S 5]{BeL94}. A simple equivariant perverse sheaf in $\mathcal M_m$ belongs to $\{ \mathsf{IC}_\la^{\mathrm{Spr}} \}_{\la \in \tP^+_n}$. We set $\mathcal M := \bigoplus_{m > 0} \mathcal M_m$. The sheaf $\mathsf{IC}^{\mathrm{Spr}}_{m} \in \mathcal M_m$ ($m \in \Z_{\ge 0}$) is the skyscraper sheaf supported on $0 \in \mathcal N_{m}$. 

We consider the subdivision $n=n_1 +n_2$ and consider two subgroups $G_i := \mathop{GL}(n_i,\C)\subset G=\mathop{GL}(n,\C)$ $(i=1,2)$ such that $T_i := (T \cap G_i)$ and $B_i := (B \cap G_i)$ $(i=1,2)$ satisfies $T = T_1 \cdot T_2$, $\varepsilon_j (T_1) = 1$ for $1 \le j \le n_1$, and $\varepsilon_j (T_2) = 1$ for $n_1 < j \le n$. We have a parabolic subgroup $P = (G_1 \times G_2)B = (G_1 \times G_2) \ltimes U \subset G$. We set $M_{n_1,n_2} := \mathrm{Lie} \, P \subset M_n$. For $i=1,2$, we set $\mathcal N_i \subset M_{n_i}$ to be the nilpotent cone of $M_{n_i}$. For $\mathcal F_i \in \mathcal M_{n_i}$ ($i=1,2$), we form a complex $\mathcal F_1 \boxtimes \mathcal F_2$ supported on
$$\mathcal N_{n_1} \times \mathcal N_{n_2} \subset M_{n_1} \oplus M_{n_1}.$$
We have a $P$-equivariant surjection $\eta: M_{n_1,n_2} \rightarrow M_{n_1} \oplus M_{n_1}$. We have two smooth maps
$$G \times ^P M_{n_1,n_2} \stackrel{\mathtt q}{\longleftarrow} G \times M_{n_1,n_2} \stackrel{\mathtt p}{\longrightarrow} M_{n_1,n_2}, \hskip 5mm \text{where $\mathtt q$ is a principal $P$-bundle}.$$
We have a perverse sheaf $\mathscr F_1 \odot \mathscr F_2$ on $G \times ^P M_{n_1,n_2}$ such that $\mathtt q^* (\mathscr F_1 \odot \mathscr F_2) \cong \mathtt p^* \eta^* ( \mathscr F_1 \boxtimes \mathscr F_2 )[2 n_1 n_2]$ and it is $G$-equivariant. Using this and a proper morphism $\mu_P: G \times^P M_{n_1,n_2} \to M_{n}$, we set
\begin{equation}
\mathscr F_1 \star \mathscr F_2 := (\mu_P)_*( \mathscr F_1 \odot \mathscr F_2).\label{eqn:conv}
\end{equation}

\begin{thm}[Lusztig, Hotta-Shimomura]\label{thm:ind} We have a monoidal functor $\star : \mathcal M \times \mathcal M \longrightarrow \mathcal M$ extending {\rm (\ref{eqn:conv})} that respects the class of perverse sheaves. For each $\la \in \tP_{n_1}^+$ and $\mu \in \tP_{n_2}^+$, we have
\begin{equation}
\mathsf{IC}_\la^{\mathrm{Spr}} \star \mathsf{IC}_\mu^{\mathrm{Spr}} \cong \bigoplus_{\nu \in \tP^+_{n}} \left( \mathsf{IC}_\nu^{\mathrm{Spr}} \right)^{\oplus c_{\la,\mu}^{\nu}},\label{eqn:LRcoeff}
\end{equation}
where $c_{\la,\mu}^{\nu}$ is the Littlewood-Richardson coefficient.
\end{thm}

\begin{proof}[Sketch of proof]
The monoidal structure follows from \cite[Proposition 4.2]{Lus85} since the both of $(\mathscr F_1 \star \mathscr F_2 ) \star \mathscr F_3$ and $\mathscr F_1 \star ( \mathscr F_2  \star \mathscr F_3)$ for $\mathscr F_i \in \mathcal M_{m_i}$ ($i=1,2,3$) with $n=m_1+m_2+m_3$ are realized as a common three-step parabolic induction.

The functor $\star$ preserves the class of perverse sheaves by \cite[Proposition 4.5]{Lus84} (see also \cite[\S 4]{Lus85}). Let $\gamma = \sum_{i=1}^n \gamma_i\varepsilon_i \in \tP_n^+$ and let $B \subset P_{\gamma} \subset G$ be a standard parabolic subgroup whose Levi subgroup is $\prod_{i=1}^n \mathop{GL}(\gamma_i,\C)$. It defines a map $\mu_\gamma : T^*G/P_{\gamma} \rightarrow \mathcal N$. In view of \cite{BM81} and \cite[\S 1, \S 8]{HSh}, we have
\begin{equation}
( \mu_\gamma )_*\C [\dim \, T^*G/P_{\gamma}] \cong \bigoplus_{\la \in \tP^+_n} \C^{k^{\la}_{\gamma}} \boxtimes \mathsf{IC}_\la^{\mathrm{Spr}},\label{eqn:HSh}
\end{equation}
where $k^{\la}_{\gamma}$ is the number of standard $\gamma$-tableaux of shape $\la$. Note that $k^\la_\gamma \neq 0$ only if $\la \unlhd \gamma$, and $k^{\gamma}_{\gamma}=1$ for each $\gamma \in \tP_n^+$. By \cite[Proposition 4.2]{Lus85} and (\ref{eqn:HSh}), we deduce
\begin{equation}
\mathsf{IC}^{\mathrm{Spr}}_{\gamma_1} \star \cdots \star \mathsf{IC}_{\gamma_n}^{\mathrm{Spr}} \cong ( \mu_\gamma )_*\C [\dim \, T^*G/P_{\gamma}] \cong \bigoplus_{\la \unlhd \gamma} \C^{k^{\la}_{\gamma}} \boxtimes \mathsf{IC}_\la^{\mathrm{Spr}}\label{eqn:indIC}
\end{equation}
for each $\gamma \in \tP_n^+$. As $k^{\la}_{\gamma}$ $(\la, \gamma \in \tP^+_n)$ are the Littlewood-Richardson coefficients \cite[I \S 9]{Mac95}, we conclude (\ref{eqn:LRcoeff}) by the monoidality.
\end{proof}

\section{Main Results}
Keep the setting of the previous section.

\subsection{Statements}\label{subsec:statement}

\begin{thm}\label{thm:main}
For each Hessenberg function $\tth \in \mathbb H_n$, we have
$$X_\tth (t) = t^{- |\Psi_{\tth}|} \sum_{\la \in \tP^+_n} s_\la \cdot t^{i/2} \dim \, H^{i} ( V^\la_{\Psi_\tth} )\in \La,$$
where $V^\la_{\Psi_\tth}$ is borrowed from Corollary \ref{cor:dec}.
\end{thm}

\begin{ex}\label{ex:SW}
Consider $\tth \in \mathbb H_3$ given as $\tth (1)=2,\tth(2)=\tth(3)=3$. We have
$$\sum_{\la \in \tP^+_n,i\in \Z} s_\la \cdot t^{i/2} \dim \, H^{i} ( V^\la_{\Psi_\tth} ) = t^2s_{21}+t^2(t^{-1}+2+t) s_{1^3}.$$
This coincide with the chromatic quasisymmetric function $t s_{21}+(1+2t+t^2) s_{1^3}$ borrowed from \cite[Example 3.2]{SW16} up to a grading shift.
\end{ex}

\begin{rem}
In our proof of Theorem \ref{thm:main}, we use \cite{AB21} and \cite[\S 6]{SP}. Still, the whole proof is independent of \cite{BC18,AB22} and other references after \cite{SW16}.
\end{rem}

We recover some well-known properties of chromatic quasisymmetric functions:
\begin{cor}[Shareshian-Wachs \cite{SW16} Theorem 6.3 and Corollary 2.8]\label{cor:qs}
For each Hessenberg function $\tth \in \mathbb H_n$, the function $X_\tth (t)$ is Schur-positive and has palindromic coefficients.
\end{cor}

\begin{proof}
The Schur positivity is a direct consequence of Corollary \ref{cor:dec}. The Verdier duality $\mathbb D$ commutes with the proper pushforward $( \mathsf{m}_\Psi )_* = ( \mathsf{m}_\Psi )_!$ \cite[(8.3.13)]{CG97}, and hence we have
$$\mathbb D \left( ( \mathsf{m}_\Psi )_* \mathbb C_{\mathscr X _{\Psi_\tth}} \right) \cong ( \mathsf{m}_\Psi )_* \mathbb D \mathbb C_{\mathscr X _{\Psi_\tth}} \cong ( \mathsf{m}_\Psi )_* \mathbb C_{\mathscr X _{\Psi_\tth}}[2 \dim \, \mathscr X_{\Psi_\tth}] \hskip 6mm \text{by \cite[(8.3.8)]{CG97}.}$$
In view of $\mathbb D \mathsf{IC}_\lambda \cong \mathsf{IC}_\lambda$, we deduce $( V_{\Psi_{\tth}}^\lambda )^* \cong V_{\Psi_{\tth}}^\lambda [2 \dim \, \mathscr X_{\Psi_\tth}]$, which implies
$$\gdim \, ( V_{\Psi_{\tth}}^\lambda )^* = t^{-\dim \mathscr X_{\Psi_\tth}} \gdim \, V_{\Psi_{\tth}}^\lambda$$
for each $\lambda \in \tP^+_n$ as required.
\end{proof}
We set $X_\tth := \pi_n( X_\tth(1) )$ for each $\tth \in \mathbb H_n$. For each $0 \le h < i \le n$, we set
$$\mathtt{S}_{i,h} := (s_{h-1} \cdots s_0 \cdots s_{i+1} s_i + s_{h-2} \cdots s_0 \cdots s_i + \cdots + s_{i+1}s_i+ s_i + 1 ) \in \C \tSym_n,$$
where the indices are understood to be modulo $n$, and $\mathtt{S}_{i,h}$ has $(n-i+h+1)$-terms.

\begin{thm}\label{thm:alg}
For each Hessenberg function $\tth \in \mathbb H_n$ corresponding to the root ideal $\Psi$, we form
$$\widetilde{X}_\tth (q) := \mathtt{S}_{1,h_1(\Psi)} \cdot \mathtt{S}_{2,h_2(\Psi)} \cdots \mathtt{S}_{n,h_n(\Psi)} e^{\varLambda_n} \in \C \tP_\af \hskip 5mm \text{where we set} \hskip 3mm q = e^{- \delta}.$$
By identifying $e^{\varepsilon_i}$ with $X_i$ for $1 \le i \le n$, we have
$$\widetilde{X}_\tth (q) \MID_{q=1,e^{\varLambda_0} = 1} = X_\tth \in \La^{(n)}.$$
In addition, the coefficient of the monomial $X^{\gamma}$ in $X_\tth$ is 
\begin{equation}
\dim \, H^{\bullet} ( \imath_{\gamma}^* (\mathsf{m}_\Psi)_* \C ) =  | \mathsf{m}_\Psi^{-1} ([\gamma])^{( T \times \Gm )} | \hskip 5mm \text{for each} \hskip 3mm \gamma \in \tP_n.\label{eqn:coeff}
\end{equation}
\end{thm}

\begin{ex}
In the same setting as in Example \ref{ex:SW}, we have
\begin{align*}
\widetilde{X}_\tth (q) &= (s_2s_1+s_1+1)(s_2+1)(s_0+1) e^{\varLambda_3} = (s_2s_1+s_1+1)(s_2+1)( q^{-1} e^{\epsilon_1 + \varLambda_2}+ e^{\varLambda_3})\\
& = (s_2s_1+s_1+1)(q^{-1} e^{2\epsilon_1+\epsilon_2} + q^{-1} e^{2\epsilon_1+\epsilon_3} + 2 e^{\varpi_3}) e^{\varLambda_0} = ( q^{-1}m_{21} + 6 m_{1^3} ) e^{\varLambda_0}.
\end{align*}
\end{ex}

\begin{cor}\label{cor:Hess}
For each $\tth \in \mathbb H_n$, we have a $G$-action on $H^{\bullet} ( \mathscr X_\Psi, \C )$ whose $T$-character is $X_\tth$.
\end{cor}

\begin{proof}
By Corollary \ref{cor:dec}, we have $H^{\bullet} ( \mathscr X_\Psi, \C ) \cong \bigoplus_{\la \in \tP_{n}^+} V_{\Psi}^\la \boxtimes H^{\bullet} ( \mathsf{IC}_\la )$. The $G$-action exists on its RHS by \cite{MV07}, and the $T$-character of $H^{\bullet} ( \mathscr X_\Psi, \C )$ is $X_\tth$ by the comparison of (\ref{eqn:coeff}) and Theorem \ref{thm:lusztig}.
\end{proof}

\begin{cor}[see e.g. \cite{DK22} Corollary 11.6]\label{cor:EP}
For each unit interval graph $\Gamma$ with $n$-vertices, we have a root ideal $\Psi$ whose Hessenberg function $\tth$ satisfies $\Gamma = \Gamma_\tth$. The Euler-Poincar\'e characteristic $\chi ( \mathscr X_\Psi )$ of $\mathscr X_\Psi$ is the number of proper colorings of $\Gamma$ with $n$-colors. Explicitly, we have
\begin{equation}
\left. X_{\tth} \right|_{X_1 =\cdots = X_n = 1} = \dim \, H^{\bullet} ( \mathscr X_\Psi, \C ) = \chi ( \mathscr X_\Psi ) = \prod_{j=1}^n (n+1-j+h_j(\Psi)).\label{eqn:EP}
\end{equation}
\end{cor}

\begin{proof}
By Corollary \ref{cor:dec} and Theorem \ref{thm:lusztig}, we deduce
\begin{align*}
\text{LHS } \text{of (\ref{eqn:EP})} \stackrel{\small (\ref{eqn:coeff})}{=} & \sum_{\gamma} \dim \, H^{\bullet} ( \imath_{\gamma}^* (\mathsf{m}_\Psi)_* \C )  = \sum_{\la, \gamma} \dim \, V_\Psi^\la \boxtimes H^{\bullet} ( \imath_{\gamma}^* \mathsf{IC}_\la )\\
& = \sum_{\la} \dim \, V_\Psi^\la \boxtimes H^{\bullet} ( \mathsf{IC}_\la ) = \dim \, H^{\bullet} ( \bigoplus_\la V_\Psi^\la \boxtimes\mathsf{IC}_\la )\\
& = \dim \, H^{\bullet} ( (\mathsf{m}_\Psi)_* \C ) = \dim \, H^{\bullet} ( \mathscr X_\Psi, \C ).
\end{align*}
We have $H^{\mathrm{odd}} ( \mathbb P^{\bullet}, \C ) \equiv 0$ and $\chi ( \mathbb P^m) = m+1$. Since $\mathscr X_\Psi^{(j)}$ is a fibration over $\P^{n-j+h_j (\Psi)}$ with its fibers $\mathscr X_\Psi^{(j+1)}$ ($1 \le j \le n$), we find $H^{\mathrm{odd}} ( \mathscr X_\Psi^{(j)}, \C ) \equiv 0$ and
\[
	\chi ( \mathscr X_\Psi^{(j)} ) = \dim \, H^{\bullet} ( \mathscr X_\Psi^{(j)}, \C ) = (n+1-j+h_j(\Psi))  \dim \, H^{\bullet} ( \mathscr X_\Psi^{(j+1)}, \C )
\]
for $1 \le j \le n$ by induction using the degeneration of the Leray spectral sequences from the base case $\mathscr X_\Psi^{(n+1)} = \mathrm{pt}$ and $H^{\neq 0} ( \mathrm{pt}, \C ) \equiv 0$. This yields
$$\dim \, H^{\bullet} ( \mathscr X_\Psi, \C ) = \chi (\mathscr X_\Psi ) = \prod_{j=1}^n (n+1-j+h_j(\Psi))$$
by induction as desired.
\end{proof}

\subsection{Proof of Theorem \ref{thm:main}}
We set the symmetric function in the statement of Theorem \ref{thm:main} as $X_\tth'(t)$.

For each $\Psi \subset \Delta^+$, we set $\mathscr S_\Psi := ( \mu_\Psi )_*\C _{T^*_\Psi X}$. By Theorem \ref{thm:X-Psi} and Corollary \ref{cor:Pdec}, we have
$$\mathscr S_\Psi \cong \bigoplus_{\la \in \tP^+_n} V^\la_\Psi \boxtimes \mathsf{IC}^{\mathrm{Spr}}_\la \hskip 5mm V^\la_\Psi \in D^b_c(\mathrm{pt}).$$

It suffices to verify that the collection $\{ t^{-|\Psi|} \gdim V^\la_\Psi\}_{\la \in \tP^+_n}$ satisfies the criterion of Theorem \ref{thm:AN}. To this end, we assume that the assertion hold for  $\bigsqcup_{n'< n}\mathbb H_{n'}$. The initial case $n=1$ is trivial.

We first prove the modular law. Assume that $\tth_0,\tth,\tth_1 \in \mathbb H_n$ corresponds to the root ideals $\Psi_0 \subset \Psi \subset \Psi_1 \subset \Delta^+$ that satisfies the condition 1) of Theorem \ref{thm:AN}. By adjusting grading shifts in \cite[Proposition 6.7 and three lines above]{SP}, we have
$$\mathscr S_{\Psi} \oplus \mathscr S_{\Psi} [-2] \cong \mathscr S_{\Psi_0}[-2] \oplus \mathscr S_{\Psi_1}.$$

Since $G.\mathcal N = \Gr^{(n)}$, we deduce
\begin{align*}
(1+t) \gdim V_{\Psi}^\la & = t \gdim V_{\Psi_0}^\la + \gdim V_{\Psi_1}^\la \\
& \Leftrightarrow (1+t) t^{-|\Psi|}\gdim V_{\Psi}^\la = (  t^{-|\Psi_0|} \gdim V_{\Psi_0}^\la ) +  t (  t^{-|\Psi_1|} \gdim V_{\Psi_1}^\la )	
\end{align*}

for each $\la \in \tP^+_n$ in this case. Therefore, we conclude that $X'_\tth(t)$ satisfies the modular law.
\medskip

We verify the factorization law. Assume that $\Gamma_\tth = \Gamma_{\tth_0} \sqcup \Gamma_{\tth_1}$ for $\tth,\tth_0,\tth_1 \in \mathbb H$. We have $n=n_1+n_2$ such that $\tth \in \mathbb H_n$, $\tth_1 \in \mathbb H_{n_1}$, and $\tth_2 \in \mathbb H_{n_2}$. Let $\Psi$ (resp. $\Psi_1,\Psi_2$) be the root ideal corresponding to $\tth$ (resp. $\tth_1,\tth_2$). We have $\tth (i) = \tth_1(i)$ $(i \le n_1)$, or $\tth_2(i-n_1) +n_1$ $(i > n_1)$. Now we use the setting of \S \ref{subsec:ind} with our choices of $n_1, n_2$ and $P$. We have a factorization of $\mu_\Psi$:
$$T^*_\Psi X \stackrel{\mu^\flat_\Psi}{\longrightarrow} G \times^P M_{n_1,n_2} \stackrel{\mu_P}{\longrightarrow} M_n.$$
The restriction of $\mu^\flat_\Psi$ fits into the following Cartesian diagram:
$$
\xymatrix{
T^*_\Psi X & P \times ^B n ( \Psi ) \ar@{_{(}->}[l]\ar[rr]^{\mu^\flat_\Psi |_{(P \times ^B n ( \Psi ))}}\ar[d]_{\widetilde{\eta}} & & M_{n_1,n_2}\ar[d]^{\eta}\\
& T^*_{\Psi_1} X \times T^*_{\Psi_2} X \ar[rr]^{\mu_{\Psi_1} \times \mu_{\Psi_2}}& & M_{n_1} \oplus M _{n_2}},$$
where the vertical map is a smooth (affine) fibration with relative dimension $n_1 n_2$, and the $U$-actions in the bottom line is trivial. Thus, if we have
$$(\mu_{\Psi_1} \times \mu_{\Psi_2} )_* \C \cong \bigoplus_{\la \in \tP^+_{n_1},\mu \in \tP^+_{n_2}} \left( V_{\Psi_1}^{\la} \otimes V_{\Psi_2}^{\mu} \right) \boxtimes \left( \mathsf{IC}_\la^{\mathsf{Spr}} \boxtimes \mathsf{IC}_\mu^{\mathsf{Spr}} \right),$$
then we have $( \mu^\flat_\Psi)_* \C \cong \bigoplus_{\la \in \tP^+_{n_1},\mu \in \tP^+_{n_2}} \left( V_{\Psi_1}^{\la} \otimes V_{\Psi_2}^{\mu} \right) \boxtimes \left( \mathsf{IC}_\la^{\mathsf{Spr}} \odot \mathsf{IC}_\mu^{\mathsf{Spr}} \right)[-2n_1n_2]$ by the smooth base change. In particular, we have
\begin{align*}
( \mu_\Psi )_* \C [2n_1n_2] & \cong (\mu_P)_* \bigr( \bigoplus_{\la \in \tP^+_{n_1},\mu \in \tP^+_{n_2}} \left( V_{\Psi_1}^{\la} \otimes V_{\Psi_2}^{\mu} \right) \boxtimes \left( \mathsf{IC}_\la^{\mathsf{Spr}} \odot \mathsf{IC}_\mu^{\mathsf{Spr}} \right)\bigr)\\
& \cong \bigoplus_{\la \in \tP^+_{n_1},\mu \in \tP^+_{n_2}} \left( V_{\Psi_1}^{\la} \otimes V_{\Psi_2}^{\mu} \right) \boxtimes \left( \mathsf{IC}_\la^{\mathsf{Spr}} \star \mathsf{IC}_\mu^{\mathsf{Spr}} \right) \\
& \cong \bigoplus_{\la \in \tP^+_{n_1},\mu \in \tP^+_{n_2}, \nu \in \tP^+_n} \left( V_{\Psi_1}^{\la} \otimes V_{\Psi_2}^{\mu} \right) \boxtimes \left( \mathsf{IC}_\nu^{\mathsf{Spr}} \right)^{\oplus c_{\la,\mu}^{\nu}}.
\end{align*}
This grading shift in this diagram, that is $t^{-n_1n_2}$, agrees with our adjustment by $|\Psi\setminus (\Psi_1 \cup \Psi_2)| = n_1 n_2$. Applying Theorem \ref{thm:ind}, we have
\begin{align*}
	X_\tth' (t) & = \sum_{\la,\mu,\nu \in \tP^+} c_{\la,\mu}^{\nu} s_{\nu} \cdot (\gdim \, V^{\Psi_1}_\la)(\gdim \, V^{\Psi_2}_\mu) \\
	& = \sum_{\la,\mu \in \tP^+} s_{\la} s_{\mu} \cdot (\gdim \, V^{\Psi_1}_\la)( \gdim \, V^{\Psi_2}_\mu) = X_{\tth_1}' (t) \cdot X_{\tth_2}' (t).
\end{align*}

This proves the factorization law for $X'_\tth(t)$.

\medskip

Finally, we compute the value of $X_{\bullet}(t)$ for a complete graph with $n$-vertices. The corresponding root ideal $\Psi$ is empty. It follows that $\mathscr X_\Psi = X$ and $\mathrm{Im} \, \mu_\Psi = \{0\} \subset \mathcal N$. Thus, we find $V_\Psi^\la = 0$ except for $\la = (1^n)$ and  $V_\Psi^{(1^n)}= H^{\bullet} ( X,\C )$. Using these, we conclude $X'_\tth (t) =[n]_t! \cdot e_n$ by $s_{1^n} = e_n$ and
$$\dim \, H^{\bullet} ( X,\C ) = [n]_t! \hskip 5mm \text{\cite[Proposition 6.7.17 and its proof]{CG97}}.$$

Thanks to Theorem \ref{thm:AN}, we conclude that $X_\tth(t) = X'_\tth(t)$ for $\tth \in \mathbb H_n$ as required.

\subsection{Proof of Theorem \ref{thm:alg}}
By Theorem \ref{thm:main}, Theorem \ref{thm:lusztig} 2), and Corollary \ref{cor:dec}, we have equalities
\begin{align}\nonumber
X_\tth = \sum_{\la \in \tP^+_n} s_\la \dim \, V_\Psi^{\la} & = \sum_{\la \in \tP^+_n, \gamma \in \tP_n} X^{\gamma} \dim \,( V_\Psi^{\la} \boxtimes H^\bullet ( \imath^*_{\gamma}\mathsf{IC}_\la)) \\
& = \sum_{\la \in \tP^+_n, \gamma \in \tP_n} X^{\gamma} \dim \,H^\bullet ( \imath^*_{\gamma}(\mathsf{m}_\Psi)_* \C ).\label{eqn:Xtofib}
\end{align}

Since $\mathsf{m}_\Psi$ is proper, we apply the proper base change to obtain
\begin{equation}
\imath_{\gamma}^* (\mathsf{m}_\Psi)_* \C \cong \imath_{\gamma}^* (\mathsf{m}_\Psi)_! \C \cong H^{\bullet}_c ( \mathsf{m}_\Psi^{-1} ([\gamma]),\C ).\label{eqn:bc}
\end{equation}
Since $[\gamma]$ is $(T \times \Gm)$-fixed, $\mathsf{m}_\Psi^{-1} ([\gamma])$ acquires the action of $(T \times \Gm)$. By \cite[Chap V\!I\!I]{Kum02},
$$\widetilde{G}(\!(z)\!) [\bv_{\varLambda_i}] \subset \P ( L ( \varLambda_i ))$$
is a partial affine flag variety, and its $(T \times \Gm)$-fixed points is $\tSym [\bv_{\varLambda_i}]$. In view of (\ref{eqn:ind-def}), we conclude $| \mathscr X_\Psi^{(T\times \Gm)} | < \infty$. This set remains the same if we replace $(T\times \Gm)$ with its generic one-dimensional subtorus $\Gm^{1}$ that we fix now. Then, $R:= H^{\bullet}_{\Gm^{(1)}} (\mathrm{pt},\C)$ is a polynomial ring of one variable.

By the localization theorem (see \cite[Chap. 7, Theorem 1.1]{AF23}), we find that the composition map
\begin{equation}
H^{\bullet}_{\Gm^{1}} ( \mathscr X_\Psi,\C ) \longrightarrow \bigoplus_{\gamma \in \tP_n} H^{\bullet}_{\Gm^{1}} ( \mathsf{m}_\Psi^{-1} ([\gamma]),\C ) \longrightarrow \bigoplus_{\gamma \in \tP_n} H^{\bullet}_{\Gm^{1}} ( \mathsf{m}_\Psi^{-1} ([\gamma])^{\Gm^{1}}, \C )\label{eqn:loc}
\end{equation}
is an isomorphism up to scalar extension to $\mathrm{Frac}(R)$, the fraction field of $R$.  By construction, $\mathscr X_\Psi$ admits a $\Gm^1$-stable affine pavings. It follows that $H^{\bullet}_{\Gm^{1}} ( \mathscr X_\Psi,\C )$ is free of rank
\begin{equation}
\dim \, H^{\bullet} ( \mathscr X_\Psi,\C ) = |\mathscr X_\Psi^{\Gm^{1}}| = \sum_{\gamma \in \tP_n} |\mathsf{m}_\Psi^{-1} ([\gamma])^{\Gm^{1}}|\label{eqn:fixpoints}
\end{equation}
over $R$. Here, (\ref{eqn:bc}) implies that
\begin{equation}
H^{\bullet}_{\Gm^1} ( \mathsf{m}_\Psi^{-1} ([\gamma]),\C ) \cong \bigoplus_{\la \in \tP^+_n,\gamma \in \tP_n} V_\Psi^\la \boxtimes H^{\bullet}_{\Gm^1} ( \imath^*_{\gamma} \mathsf{IC}_\la ),\label{eqn:loc-sum}
\end{equation}
and $\imath^*_{\gamma} \mathsf{IC}_\la$ is a complex of vector spaces. We have a Leray spectral sequence
\begin{equation}
E_2^{q,p}:=H^{q}_{\Gm^{(1)}} (\mathrm{pt},\C) \otimes H^{p} ( \imath^*_{\gamma} \mathsf{IC}_\la ) \Rightarrow H^{q+p}_{\Gm^1} ( \imath^*_{\gamma} \mathsf{IC}_\la )\label{eqn:Leray}
\end{equation}
for $\la \in \tP_n^+, \gamma \in \tP_n$. By \cite[(4.2)]{Lus81a} and $H^{\mathrm{odd}}_{\Gm^{(1)}} (\mathrm{pt},\C) \cong H^{\mathrm{odd}} (\mathbb{P}^{\infty},\C) \equiv 0$, we find that the spectral sequence (\ref{eqn:Leray}) degenerates at the $E_2$-stage. It follows that
$$\mathrm{rank}_{R} H^{\bullet}_{\Gm^1} ( \imath^*_{\gamma} \mathsf{IC}_\la ) = \dim \, H^{\bullet} ( \imath^*_{\gamma} \mathsf{IC}_\la ).$$
Summing up with respect to (\ref{eqn:loc-sum}), we conclude
$$\mathrm{rank}_{R} H^{\bullet}_{\Gm^1} ( \mathsf{m}_\Psi^{-1} ([\gamma]),\C ) = \dim \, H^{\bullet} ( \mathsf{m}_\Psi^{-1} ([\gamma]),\C ).$$

The second map of (\ref{eqn:loc}) is a direct sum of
$$H^{\bullet}_{\Gm^{1}} ( \mathsf{m}_\Psi^{-1} ([\gamma]),\C ) \rightarrow H^{\bullet}_{\Gm^{1}} ( \mathsf{m}_\Psi^{-1} ([\gamma])^{\Gm^{1}}, \C ),$$
that is an inclusion on the $R$-torsion free part (again by the localization theorem). Thus, we have
$$|\mathsf{m}_\Psi^{-1} ([\gamma])^{\Gm^{1}}| \ge \mathrm{rank}_{R} H^{\bullet}_{\Gm^1} ( \mathsf{m}_\Psi^{-1} ([\gamma]),\C ) \hskip 10mm \gamma \in \tP_n.$$
In order that (\ref{eqn:loc}) to be an isomorphism up to scalar extension to $\mathrm{Frac} ( R )$, we have necessarily
\begin{equation}
|\mathsf{m}_\Psi^{-1} ([\gamma])^{\Gm^{1}}| = \mathrm{rank}_{R} H^{\bullet}_{\Gm^1} ( \mathsf{m}_\Psi^{-1} ([\gamma]),\C ) = \dim \, H^{\bullet} ( \mathtt{m}_\Psi^{-1} ([\gamma]),\C )\label{eqn:fibeq}
\end{equation}
for each $\la \in \tP_n^+$ and $\gamma \in \tP_n$ by (\ref{eqn:fixpoints}). Now we verify  (\ref{eqn:coeff}) as:
\begin{align}\nonumber
X_\tth & \stackrel{\scriptsize (\ref{eqn:Xtofib})}{=} \sum_{\la \in \tP^+_n, \gamma \in \tP_n} X^{\gamma} \dim \,H^\bullet ( \imath^*_{\gamma}(\mathsf{m}_\Psi)_* \C ) \\
& \stackrel{\scriptsize (\ref{eqn:bc})}{=} \sum_{\gamma \in \tP_n} X^{\gamma} \dim H^\bullet ( \mathsf{m}_\Psi^{-1} ([\gamma],\C))  \stackrel{\scriptsize (\ref{eqn:fibeq})}{=} \sum_{\gamma \in \tP_n} X^{\gamma} \cdot |\mathsf{m}_\Psi^{-1}([\gamma])^{\Gm^1}|.\label{eqn:XTcount}
\end{align}

In the construction of $\mathscr X_\Psi$, we have applied the action of $\tbG_i$ inductively to $\mathscr X_\Psi^{(i+1)}$ from $i=n$. Its effect is seen in $\P ( L ( \varLambda_i ))$, and makes $\mathscr X_\Psi^{(i)}$ into the fiber bundle over a projective space (isomorphic to $\tbG_i /(\tbG_i\cap \tbG_{i+1})$) with its fiber $\mathscr X_\Psi^{(i+1)}$. The $\Gm^1$-fixed counterpart of $\tbG_i$ is the elements in
$$\Sym^{(i)}:=\left< s_j \mid i \le j < n, 0 \le j < h_i ( \Psi ) \right> \subset \tSym \hskip 5mm \text{for} \hskip 5mm 1 \le i \le  n.$$
The $\tbG_{i+1}$-action permutes the $\Gm^1$-fixed points in $\mathscr X_\Psi^{(i+1)}$. Thus, the $\Gm^1$-fixed points in $\mathscr X_\Psi^{(i)}$ is obtained as the translation of the $\Gm^1$-fixed points of $\mathscr X_\Psi^{(i+1)}$ by the application of
\begin{equation}
\{ 1,s_i,s_{i+1}s_i,\ldots,s_{h_i ( \Psi) -1}\cdots s_i\} = \Sym^{(i)}/(\Sym^{(i)}\cap \Sym^{(i+1)}). \label{eqn:refs}
\end{equation}
Let $\mathsf S^i$ denote the set of elements in (\ref{eqn:refs}). The set of $\Gm^1$-fixed points of $\mathscr X^{(i)}_\Psi$ is in bijection with
$$\mathsf S^{\ge i}:= \mathsf S^i \times \mathsf S^{i+1} \times \cdots \times \mathsf S^{n}.$$

If $[\gamma] = w[\bv_{\varLambda_n}]$ ($w \in \tSym_n$), then we have $\mathsf{m}_\Psi^{-1}([\gamma])^{\Gm^1} \cong \{\{x_j\}_{j=1}^n \in \mathsf S^{\ge 1} \mid x_1 x_{2} \cdots x_n = w \Sym_n\}$ by $\mathsf{Stab}_{\tSym_n} [\bv_{\varLambda_n}] = \Sym_n$. Since we have
$\left. e^{w\varLambda_n} \right|_{q=1,e^{\varLambda_0}=1} = X^{\gamma}$, we conclude
\begin{equation}
\mathsf{m}_\Psi^{-1}([\gamma])^{\Gm^1} \cong \{\{x_j\}_{j=1}^n \in \mathsf S^{\ge 1} \mid \left. e^{x_1 x_{2} \cdots x_n\varLambda_n} \right|_{q=1,e^{\varLambda_0}=1} = X^{\gamma}\}.\label{eqn:Tcount}
\end{equation}
Substituting (\ref{eqn:Tcount}) into (\ref{eqn:XTcount}), we conclude
\begin{align*}
X_\tth & = \sum_{\gamma \in \tP_n} X^{\gamma} |\mathsf{m}_\Psi^{-1}([\gamma])^{\Gm^1}|  = \sum_{\{x_j\}_{j=1}^n \in \mathsf S^{\ge 1}} \left. e^{x_i x_{i+1} \cdots x_n\varLambda_n} \right|_{q=1,e^{\varLambda_0}=1} \\
& = \left. \bigl(\mathtt{S}_{1,h_1(\Psi)} \cdots \mathtt{S}_{n,h_n(\Psi)} e^{\varLambda_n} \bigr) \right|_{q=1,e^{\varLambda_0}=1}
\end{align*}
as required.

\medskip

{\small
{\bf Acknowledgement:} This work was supported in part by JSPS KAKENHI Grant Number JP24K21192. The author thanks Tatsuyuki Hikita for discussions.}

{\footnotesize
\bibliography{ref}
\bibliographystyle{hplain}}

\end{document}